\documentclass{amsart}
\usepackage{verbatim,amssymb,amsmath,amscd,latexsym,amsbsy,mathrsfs} 
\usepackage{graphicx}
\usepackage[all]{xy}
 \input{xy}
 \xyoption{all}
\usepackage{tikz-cd}
\usepackage{color}
\usepackage{mathrsfs}

\newtheorem{thm}{Theorem}[section]
\newtheorem{cor}[thm]{Corollary}

\newtheorem{prop}[thm]{Proposition}

\newtheorem{lemma}[thm]{Lemma}
\newtheorem{rmk}[thm]{Remark}

\newtheorem{ex}[thm]{Example}

\DeclareMathOperator*{\Ind}{Ind }
\DeclareMathOperator*{\Res}{Res }

\DeclareMathOperator*{\tr}{tr}

\newcommand{\cL}{\mathcal{L}}

\newcommand {\V} {\mathcal{V}}

\newcommand {\C} {{\mathbb C}}

\newcommand {\Z} {{\mathbb Z}}

\newcommand {\G} {\widehat{G}}

\newcommand {\OO} {{\mathcal O}}

\newcommand {\W} {\mathcal{W}}

\begin{document}
\title{Residues of connections and the Chevalley-Weil formula for curves}
\author{
        Donu Arapura    
      }
 \thanks {Partially supported by a grant from the Simons foundation}
\address{Department of Mathematics\\
 Purdue University\\
 West Lafayette, IN 47907\\
U.S.A.}
 \maketitle
 
 \begin{abstract}
  Given a finite group of automorphisms of a compact Riemann surface,
  the Chevalley-Weil formula computes the character valued Euler
  characteristic of an equivariant line bundle. The goal of this
  article is to give a proof by computing  using residues of a
  Gauss-Manin connection.
 \end{abstract}

Let  $G$ be a finite group of automorphisms of a compact Riemann
surface $X$. 
Then $G$ will act on the space  of holomorphic $1$-forms,  more generally on the
space of sections of powers of the canonical bundle $K$, or even more generally on the cohomology
of an equivariant holomorphic line bundle $L$. The virtual representation $H^0(X,L)-H^1(X,L)$
is  determined
by its character $\chi_G(L)$, and there  are at least two ways in which one
might try to compute it. First of all, one could try to find  a formula for the character as a complex valued function on 
the group. The holomorphic Lefschetz fixed point theorem of Atiyah and Bott \cite{ab} gives a formula for $\chi_G(L)$
in this form:
$$\chi_G(L)(g)= \sum_p \frac{\nu_p(g)}{1-\tau_p(g)}, \>g\not=1$$
where the sum runs over fixed points of $g$,
and $\tau_p$ and $\nu_p$ are the characters of the action of the isotropy groups $G_p$ 
 on $T_p^*$ and $L_p$.
The second, and arguably more useful type of formula, is  an  expression for  $\chi_G(L)$ as a linear combination  of 
characters of irreducible  representations, the set of which we denote
by $\G$. Chevalley and Weil \cite{cw} gave such a formula when $L$ is
a positive power of the canonical bundle.  Formulas of this type for more
general $L$ were considered
by Ellingsrud and L{\o}nsted \cite{el}.
When $G$ acts freely
on $X$,  $\chi_G(L)  $ is a multiple of the character of the
regular representation $\chi_{reg}$. More specifically,
$$\chi_G(L) =\left(\frac{1}{|G|}\deg L + 1-h\right)\chi_{reg} $$
where $h$ is the genus of $X/G$.
If the action is not free, then
$$\chi_G(L) =\left(\frac{1}{|G|}\deg L + 1-h\right)\chi_{reg}
-\sum_{\xi\in \G} m_\xi(L)\xi$$
where the coefficients $m_\xi(L)$ are given by explicit formulas depending  on $\xi,\tau_p$ and $\nu_p$.
This can be made much more explicit
when $G$ is cyclic, and we will see that  the holomorphic Lefschetz fixed
point theorem for finite order automorphisms of Riemann surfaces is a straightforward consequence of this.

Our main goal is  give a new proof of the (generalized) Chevalley-Weil
formula.
The argument is neither very efficient nor  particularly general, but we
hope that it is instructive.
We start by proving a general residue theorem, which  expresses the degree of a vector
bundle on a compact Riemann surface as minus the sum of traces of residues of a  singular holomorphic 
connection. Next given an equivariant line bundle $L$,
we decompose the direct image $\pi_*L=\bigoplus
V_\xi$ into isotypic components, where $\pi:X\to X/G$ is the projection.
The heart of the proof is the
computation of the degrees of the  bundles $V_\xi$ by using the
residue theorem applied to certain Gauss-Manin connections.
With these in hand, the formula follows easily from
Riemann-Roch on $X/G$.  
 
 \section{Characters}
We start by recalling a few facts from representation theory needed below
\cite{serre}. Let $G$ be a finite group of order $N$. Let $F(G)$
denote the ring of class functions on $G$. The vector space $F(G)$ has a basis given by
the set $\G$ of  characters of irreducible representations. In fact, $\G$ forms
an orthonormal basis with respect to the inner product
$$\langle \chi, \eta\rangle=\langle\chi ,\eta \rangle_G =
\frac{1}{N}\sum_{g\in G} \chi(g)\overline{\eta(g)}$$
Thus any element $\eta\in F(G)$ can be
expanded as a Fourier series
$$\eta= \sum_{\chi\in \G} \langle \eta,\chi\rangle \chi$$
We write
$$\chi_{reg}(g)= \sum_{\chi\in \G} \chi(1)\chi(g)=
\begin{cases}
  N &\text{if $g=1$}\\
  0 &\text{otherwise}
\end{cases}
$$
for the character of the regular representation $\C[G]$.
If $\chi\in \G$, and $M$ is a $G$-module, then the $\chi$-isotypic component
$M_\chi\subset M$ is given by $e_\chi M$, where
$$e_\chi = \frac{\chi}{N}\sum_{g\in G} \chi(g^{-1})g\in \C[G]$$
is a central idempotent.

If $H\subset G$ is a subgroup, let $\Ind:F(H)\to F(G)$ be
induction. This takes the character of an $H$-module $M$ to the character
of the induced $G$-module $\C[G]\otimes_{\C[H]} M$. In
the opposite direction we have restriction. Frobenius reciprocity says
that these operations are adjoint, i.e. if $\chi\in F(H),\eta\in F(G)$,
$$\langle\chi,\eta|_H\rangle_H= \langle \Ind\chi, \eta\rangle_G$$

\section{Statement of the Chevalley-Weil formula}

Let $X$ be a compact Riemann surface. Fix   a finite subgroup $G\subseteq
Aut(X)$ of  order $N$, and let $\pi:X\to Y=X/G$ be the quotient. For each $p\in X$, let $G_p=
\{g\in G\mid gp=p\}$ be the isotropy group. It acts faithfully on the
cotangent space $T_p^*= m_p/m_p^2$. In particular, it is cyclic of
finite order, say $N_p$. Let $\tau_p$ be the character of the
representation of $G_p$ on $T_p^*$. We can choose an isomorphism $G_p\cong \Z/N$ such that
$\tau_p(1)=e^{2\pi i/N}=\zeta$, and a local coordinate $x$ at $p$ so
that $\tau_p(1)$ acts by $x\mapsto \zeta x$.

Let $\psi:L\to X$ be a   line bundle in the geometric sense, i.e. $L$ is a complex manifold, with holomorphic projection $\psi$
satisfying the usual conditions. There is an action of $\C^*$ on $L$ preserving the fibres.
Let $Aut^{\C^*}(L)$ denote the group of holomorphic automorphisms of the manifold
$L$ which commute with the action of $\C^*$. It follows easily that an element of  $Aut^{\C^*}(L)$  preserves the zero section. Thus we have a homomorphism
$$p:Aut^{\C^*}(L)\to Aut(X)$$
which sends $\gamma$ to its restriction to the zero section.
Let 
$$Aut(X)_L= \{g\in Aut(X)\mid g^*L\cong L\}$$
denote the isotropy group of $L\in Pic(X)$.
There is an exact sequence
$$1\to \C^*\to Aut^{\C^*}(L)\xrightarrow{p} Aut(X)_L\to 1$$
\cite[lemma 3.4.1]{brion}. Let us suppose that $G\subseteq Aut(X)_L$, or equivalently that $L\in Pic(X)^G$ is invariant under $G$.
A $G$-linearization of $L$  is a splitting of
\begin{equation}
  \label{eq:G}
   1\to \C^*\to p^{-1}G\to G \to 1,
\end{equation}
 or in other words a lifting of $G$ to a subgroup of $ Aut^{\C^*}(L)$.
 A $G$-equivariant line bundle on $X$ is a  line bundle with a
 $G$-linearization. Equivariant bundles and invariant bundles are not
 quite the same thing. If $Pic_G(X)$ denotes the group of equivariant
 line bundles, then there is forgetful homomorphism
 $$Pic_G(X)\to Pic(X)^G$$
In general, it is neither injective nor surjective, but  it will
 have finite kernel and cokernel. To see this, observe that
by basic  facts about group cohomology \cite[chap IV]{brown}, the
obstruction to splitting \eqref{eq:G} lies in $H^2(G, \C^*)$.
 If the obstruction vanishes, then the set of splittings is parametrized by
 $H^1(G,\C^*)\cong Hom(G,\C^*)$. Both cohomology groups are finite.
 
Given a $G$-equivariant line bundle $L$, $G$ will act on  $L$ by bundle
automorphisms. Consequently, it will act on the space of holomorphic sections $\cL(U) = \OO_X(L)(U)$ for
any $G$-invariant set $U$. Therefore
 $G$ will act on the \v{C}ech complex $\check{C}(\mathcal{U}, \cL)$
 for any covering by $G$-invariant open Stein sets. (Note that such a
 cover exists, because we may  use the preimage of a Stein
 cover of the Riemann surface $X/G$.) Thus $G$  acts on $H^i(X,\cL)$. Let $\chi_G(L)$ denote
 the character of the virtual representation
 $$H^0(X,\cL)-H^1(X,\cL)$$
 In other words,  $\chi_G(L)$  is the character of $H^0(X,\cL)$ minus the character of
$H^1(X,\cL)$.

Here are a few examples of equivariant line bundles.

\begin{ex}
For any homomorphism $h:G\to \C^*$,
the trivial bundle $\C_X= \C\times X$ has a $G$-linearization, where $g(v, x)=(h(g)v,gx)$.
We usually take $h=1$. Then the action on sections $\OO_X(U)$ is the usual action on functions
$gf(x)= f(g^{-1}x)$
\end{ex}

\begin{ex}
The canonical bundle $K$ has a natural $G$-linearization such that the $G$-action on
 sections is the action  $\omega\mapsto g^*\omega$ on $1$-forms. There
 is an induced linearization on any power $K^{\otimes n}$.
\end{ex}

\begin{ex}
 If $D=\sum n_p p$ is a divisor, then  recall that we have a line bundle whose sheaf of sections is
 $$\OO_X(D)(U) = \{f \text{ meromorphic on } U\mid \forall p\in U, ord_p(f)\ge -n_p\}$$
 If $D$ is $G$-invariant, $\OO_X(D)$ has
 a $G$-linearlization such that the action is $gf(x) = f(g^{-1}
 x)$. All $G$-equivariant line bundles arise this way; more precisely
 $$Pic_G(X)\cong \frac{Div(X)^G}{\pi^*Princ(Y)}$$
 \cite[thm 2.3]{borne}.
\end{ex}

 If $L$ is equivariant, then  $G_p$ will act on the
 fibre $L_p$. Let $\nu_p$ denote the character of this representation.

\begin{thm}\label{thm:cw}
Let $X,G$ be as above, and let $L$ be an equivariant line bundle on
$X$.  Let $Y=  X/G$ have genus $h$. Then
\begin{equation}
  \label{eq:CW}
\chi_G(L) = \left(\frac{1}{N}\deg L + 1-h\right)\chi_{reg} - \sum_{\xi\in
  \G} m_\xi(L) \xi
\end{equation}
where
$$m_{\xi,p}(L) = \frac{\xi(1)}{N} \sum_{i=0}^{N_p-1}
  i\langle \xi|_{G_p}, \nu_p\cdot \tau_p^i\rangle_{G_p}$$
  $$m_\xi(L) = \sum_{p\in X} m_{\xi,p}(L)$$
\end{thm}

\begin{rmk}

  \begin{enumerate}
    \-
    \item This is a special case of a more general result due to Ellingsrud and
  L{\o}nsted \cite{el}. When 
  $L=K^{\otimes n}$, the formula goes back to  Chevalley and Weil \cite{cw}.
\item The sums $m_\xi(L)$ are clearly finite since $m_{\xi,p}(L)=0$
  unless $G_p\not=\{1\}$. We also note that $m_{\xi, p}(L)$ only
  depends on the orbit by the arguments of section 4. So we can rewrite
  $$m_\xi(L)=\sum_{q\in Y} |G_p|m_{\xi, p}(L)$$
  where we choose one $p\in X$ over each $q$.
  
\item The formula can be made more explicit in certain cases, see
section 5.
\end{enumerate}

  \end{rmk}

\begin{cor}\label{cor:free}
If $G$ acts freely, then 
$$\chi_G(L) = \left(\frac{1}{N}\deg L + 1-h\right)\chi_{reg}$$
In particular, $\chi_G(L)$ is a multiple of the regular representation.
\end{cor}

\section{The Residue theorem}

Let $V$ be a holomorphic vector bundle on compact Riemann surface $Y$, and let
$\V = \OO_Y(V)$ be its sheaf of holomorphic sections. Let $\Omega_Y^1$
be the sheaf of holomorphic $1$-forms.
A holomorphic connection on $V$ (or $\V$) is a $\C$-linear map
$$\nabla:\V\to \Omega_Y^1\otimes \V$$
satisfying the Leibnitz rule.
 Note that the curvature necessarily vanishes, because it locally
a matrix of  holomorphic $2$-forms. Therefore by Chern-Weil theory  \cite[chap 3, \S3]{gh},
the degree $\deg V:= \int_Yc_1(V)=0$. This provides an obstruction to the existence of such a connection.
We can avoid it by allowing singularities.
Given a finite set $S\subset Y$, let  $\Omega_Y^1(\infty S)$
($\Omega_X^1(*S)$,  resp. $\Omega_Y^1(\log S)$)  be the sheaf
of holomorphic forms on  $Y-S$ (with
poles of finite order, resp. order at most $1$, on points of $S$). Given a vector bundle
$V$ on $Y$, a $\C$-linear map
$$\nabla:\V\to \Omega_Y^1(\infty S)\otimes \V$$
satisfying the Leibnitz rule will be referred to as a singular holomorphic
connection with singularities along $S$.  We call $\nabla$ meromorphic
(resp. logarithmic) if it takes values in $\Omega_Y^1(*S)\otimes \V$ (resp. $ \Omega_Y^1(\log S)\otimes \V$).
Any vector bundle $V$ carries a meromorphic connection, namely
 $\nabla = d$  where $V|_{X-S}\cong \OO_{X-S}^n$ is a local trivialization of the underlying algebraic
 vector bundle (which exists by GAGA).
 When $V$ is a line bundle, we will give a more precise construction in lemma \ref{lemma:Gconnection}.

The standard reference for logarithmic connections is Deligne
\cite{deligne}, although we will not really need much beyond the definition of residue, which we now explain.
Given $s\in S$, and a basis $\{v_i\}$ of the stalk $\V_s$ as an
$\OO_s$-module, a singular holomorphic 
connection is determined by a matrix of $1$-forms $A=(\alpha_{ij})$ given by
$$\nabla v_i =\sum_j \alpha_{ij}\otimes v_j$$
 The forms $\alpha_{ij}$ are holomorphic in a punctured disk about
 $s$. Therefore given a local
parameter $y$ at $s$,
we can form a Laurent expansion
$$A= \sum_{n=-\infty}^\infty A_ny^ndy$$
where the coefficients are constant matrices.
In the logarithmic case, $A_{-1}$ gives  a well
defined endomorphism of the fibre $ End(V_s)$ \cite[p 78]{deligne}.
For a general singular holomorphic connection, we claim something weaker:

\begin{lemma}
 Given a singular holomorphic connection,
the trace $\tr A_{-1}\in \C$ is well defined, i.e. it
is independent of the choice of basis of $\V_s$ and local parameter.
\end{lemma}

\begin{proof}
  Suppose that $v_i'=\sum
f_{ij}v_j$ is a new basis. If $F= (f_{ij})$, then the new connection matrix is given by $A'= FAF^{-1}
+ dF\cdot F^{-1}$. Therefore $\tr A'= \tr A+ \tr(dF\cdot F^{-1})$, and so
$\tr A_{-1}'=\tr A_{-1}$. Cauchy's integral formula
shows that $\tr A_{-1}$ is independent of the local parameter.
\end{proof}

We denote $\tr A_{-1}$ by $\tr {\Res}_s(\nabla)$.

\begin{thm}\label{thm:deg}
  If $(V,\nabla)$ is vector bundle equipped with  a singular holomorphic connection
  as above, then
  $$\deg V= -\sum_{s\in S} \tr {\Res}_s(\nabla)$$
\end{thm}

\begin{rmk}
  \-
  \begin{enumerate}
  \item The theorem generalizes the well known fact that the sum of residues
    of meromorphic $1$-form $\omega$ is zero \cite[p 222]{gh}. This
    follows by applying the theorem when $V$ is trivial with
    connection $d+\omega$.
    \item When the connection is logarithmic,
  this is a special case of a theorem of Esnault and Verdier \cite[appendix B]{ev}. 
  \end{enumerate}
  
\end{rmk}
\begin{proof}
  Choose $\epsilon>0$, so that the closed disks  $D_s(\epsilon)$ of
  radius $\epsilon$ (with respect to a metric) around $s$ are disjoint and
  contained in coordinate neighbourhoods. By a standard partition of
  unity argument, we can construct a $C^\infty$
  connection $\nabla'$ on $V$ which agrees with $\nabla$ on $\bigcup_s
  D_s(\epsilon/2)$.
  The curvature $\Theta$ of $\nabla'$ must vanish on $Y-\bigcup_s
  D_s(\epsilon)$, because it coincides with the curvature of $\nabla$
  on this set. By the Chern-Weil formula \cite[chap 3, \S3]{gh}
  $$\deg V= \frac{i}{2\pi} \int_Y \tr(\Theta)= \sum_{s\in S}
  \frac{i}{2\pi} \int_{D_s(\epsilon)} \tr(\Theta)$$
Let $B_s=(\beta_{ij})$ denote the connection matrix for $\nabla'$ with respect to a
trivialization of $V$ on $D_s(\epsilon)$. Then
$$\Theta= dB_s+ B_s\wedge B_s$$
by the usual formula for curvature (see \cite[p 401]{gh} although the 
sign there differs from ours).
Therefore
$$\tr \Theta= \tr(dB_s) + \sum_{ij}\beta_{ij}\wedge\beta_{ij}= d \tr
B_s$$
Since $B_s$ is the
connection matrix of $\nabla$ on the annulus $D_s(\epsilon)- D_s(\epsilon/2)$,
$$\tr B_s = (y^{-1}\tr {\Res}_s \nabla+ f(y))dy$$
on the annulus, where $f$ is holomorphic on
$D_s(\epsilon)-\{s\}$ with zero residue at $s$.
Therefore by Stokes' and Cauchy's theorems
$$  \frac{i}{2\pi} \int_{D_s(\epsilon)} \tr(\Theta)=  -\frac{1}{2\pi i}
\int_{\partial D_s(\epsilon)} \tr(B_s)= -\tr {\Res}_s(\nabla)$$

\end{proof}

\begin{cor}
  $$\sum_{s\in S} \tr {\Res}_s(\nabla)\in \Z$$
\end{cor}

\section{Proof of Chevalley-Weil}

Let $\pi:X\to Y$ be a nonconstant holomorphic map between compact Riemann surfaces.
Suppose that $L$ is a holomorphic  line bundle on $X$ with a logarithmic  connection
$$\nabla:L\to \Omega_X^1(\log S)\otimes L$$
After increasing $S$ if necessary, we can assume that it contains all
the ramification points of $\pi$.
Then $T=\pi(S)$ will contain
 the branch points. Let $\cL = \OO_X(L)$, and $\V=\pi_*\cL$.
 The direct image
 $$\V=\pi_*\cL \xrightarrow{\pi_*\nabla} \pi_*(\Omega_X^1(\log
 S)\otimes \cL)\cong \Omega_Y^1(\log T)\otimes \V$$
 defines a  logarithmic  connection$\tilde \nabla$, which we refer to
as the Gauss-Manin connection. 

We give a local description.
Choose $q\in Y$, and let  $\{p_1,\ldots, p_m\}= \pi^{-1}(q)$. Let $n_i$ denote the ramification 
index of $p_i$. Choose local  coordinates about $p_i$ and $q$ so that
$y=x^{n_i}$. Let us also choose a local
generator $\lambda$ for $L$ at $p_i$. Let $\W_i$ be  the free $\OO_{p_i}$-module with basis $\lambda, x\lambda,\ldots, x^{n_i-1}\lambda$.
Then 
$$\V_q= \bigoplus_{i=1}^m \W_i$$
The  Gauss-Manin connection respects this decomposition. We will compute it on $\W_i$.
Let us suppose that $\nabla$ has residue $r_i$ at $p_i$. 
We can write
 \begin{equation}
 \label{eq:GM}
   \begin{split}
     \tilde \nabla(x^j\lambda) &=\frac{dx}{x} \otimes r_ix^j\lambda + \frac{dx}{x} \otimes jx^j
     \lambda +\ldots\\
     &=\frac{dy}{y} \otimes\left(\frac{r_i+j}{n_i}\right) x^j\lambda + \ldots
   \end{split}
 \end{equation}
 where the omitted terms are holomorphic.

\begin{lemma}
$$\tr {\Res}_q\tilde \nabla = \sum_{i=1}^m {\Res}_{p_i} \nabla + \frac{1}{2}\sum_{i=1}^m (n_i-1)$$

\end{lemma}

\begin{proof}
  Using \eqref{eq:GM}, we obtain
$$\tr {\Res}_q\tilde \nabla = \sum_{i=1}^m \sum_{j=0}^{n_i-1}\frac{r_i+j}{n_i}=  \sum_{i=1}^m \left( r_i+ \frac{n_i-1}{2} \right)$$

\end{proof}

\begin{cor}
 Let $R$ be the ramification divisor of $\pi$ then
$$
   \deg \V = \deg L -\frac{1}{2}\deg R
   $$
\end{cor}

\begin{proof}
 This follows from the lemma and   theorem \ref{thm:deg}.
\end{proof}

Let us now assume that $G\subseteq Aut(X)$ is a subgroup of order $N$, and that
 $Y=X/G$. The ramification index of  $p\in X$ is exactly $N_p=|G_p|$.
In particular, $p$ is ramified if and only if it is a fixed point for
some $g\not=1$. The
fibre $\pi^{-1}(\pi(p))$ is precisely the orbit $G\cdot p$.

\begin{lemma}\label{lemma:Gconnection}
  If $L$ is a $G$-equivariant line bundle on $X$, then there exists a $G$-invariant logarithmic connection
  $$\nabla:L\to \Omega_X^1(\log S) \otimes L$$
  for some invariant $S\subset X$. \end{lemma}

\begin{proof}
First suppose that $\deg L <0$, then for $m\gg 0$, $L^{\otimes - m}$ will have a nonzero section $\sigma'$.
Then $\sigma^\vee= \sigma'\otimes \ldots \otimes \sigma'$ will give an invariant section of $L^{\otimes -n}$, where $n=mN$.
The dual gives an invariant injective map
$\sigma:\OO_X(L^{\otimes n})\to\OO_X$. Then $\nabla(\lambda) =
d\sigma(\lambda^{\otimes n})/\sigma(\lambda^{\otimes n})\otimes \lambda$ will have the desired properties.
  We can do the  general case, by writing $L= L_1\otimes L_2^{-1}$,
  with $\deg L_i< 0$,
  and using the tensor product and dual connections.
\end{proof}

 The direct image $\V= \pi_*\cL$ is locally free with $G$-action. We
 can decompose
 \begin{equation} 
   \label {eq:sumV}
   \V= \bigoplus_{\xi \in \G} \V_\xi, \> \V_\xi = e_\xi \V   
 \end{equation}
We have
 $$e_\xi H^i(X,\cL) = H^i(Y, \V_\xi)$$
 Therefore by Riemann-Roch
 \begin{equation}
   \label{eq:RR}
   \langle\chi_G(L), \xi\rangle = \chi(Y, \V_\xi) = \deg \V_\xi +
 \xi(1)(1-h)
\end{equation}
where $\chi(Y, \V_\xi)$ is the Euler characteristic of $\V_\xi$.
 The proof of  theorem \ref{thm:cw} will follow from the computation of
 the degree $\deg \V_\xi$, which we now explain.

 Let us choose an invariant connection
 $$\nabla:L\to \Omega_X^1(\log S) \otimes L $$
 as in the above lemma. Then the Gauss-Manin connection $\tilde
 \nabla$ is $G$-invariant, so we can decompose $\tilde \nabla$ as a
 sum $\bigoplus \tilde \nabla_\xi$ with respect to \eqref{eq:sumV}.
 Choose $p=p_1\in T$, and let $q=\pi(p)$.
    Let  $n= N_{p}$, $m= N/n$,  and$\{p_1,\ldots, p_m\}= G\cdot p_1=
 \pi^{-1}(q)$.  The modules with connection $(\W_i, \tilde \nabla|_{\W_i})$ are now
 isomorphic, so it suffices to look at $\W=\W_1$. Let $W =
 \C\otimes_{\OO_p} \W$, and $ V=  \C\otimes_{\OO_q} \V_q$ denote
 the fibres. The first vector space $ W$ is a $\C[G_p]$-module. The basis
 vectors $x^i\lambda$ span $G_p$-submodules  $\C_{\nu\tau^i}$  with character
 $\nu\tau^i$. Therefore
 $$ W= \bigoplus_{i=0}^{n-1} \C_{\nu\tau^i}$$
  We note that $\{\nu, \nu\tau,\ldots, \nu\tau^{n-1}\}=
  \G_p$. Since $G_p$  is abelian, the $G_p$-module
  $$ W \cong \C[G_p]$$
The vector space $ V$ is the $\C[G]$-module
induced from the $G_p$-module $ W$. Consequently $ V\cong
\C[G]$. It will useful to weight the summands, by considering  the class function
\begin{equation}
  \label{eq:rho}
\rho_q = \sum_{i=0}^{n-1} \left(\frac{r+i}{n}\right) \Ind(\nu\tau^i)  
\end{equation}
We can see from \eqref{eq:GM} and the above remarks that
\begin{equation}
  \label{eq:trR}
\tr {\Res}_q \tilde\nabla_\xi = \langle \rho_q, \xi\rangle_G  
\end{equation}

\begin{proof}[Proof of theorem \ref{thm:cw}]
 It suffices to  prove that
 $$\langle \chi_G(L),\xi\rangle =\left(\frac{1}{N}\deg L + 1-h \right)\xi(1)-m_\xi(L)$$
 for every $\xi\in \G$. By comparing with \eqref{eq:RR}, we see that
 if suffices to prove that
 \begin{equation}
   \label{eq:degVxi}
 \deg \V_\xi =\frac{1}{N}(\deg L) \xi(1)-m_\xi(L)   
 \end{equation}

 By  theorem~\ref{thm:deg},
 $$\deg \V_\xi = -\sum_{q\in T} \tr {Res}_q \tilde \nabla_\xi =
 -\sum_{q\in T}
 \langle \rho_q, \xi\rangle_G$$
  Let us write
  $\rho_q = \rho_q'+\rho_q''$, where
  $$\rho_q' = \sum_{i=0}^{n-1} \left(\frac{r}{n}\right)
  \Ind(\nu\tau^i)  $$
  $$\rho_q'' = \sum_{i=0}^{n-1} \left(\frac{i}{n}\right)
       \Ind(\nu\tau^i)$$
  As noted above,
  $$\rho_q' =  \left(\frac{r}{n}\right)\Ind  W =
  \left(\frac{r}{n}\right)\chi_{reg}$$
  Therefore
  $$\langle-\rho_q', \xi\rangle = \left( \frac{r}{n}\right) \xi(1) = rm
  \frac{\xi(1)}{N}$$
  The quantities $r$ and $m$ depend on $q$, even though the notation
  does not reflect this.
  Summing over $q$, and using theorem~\ref{thm:deg} yields
  $$\sum_q rm = -\deg L$$
  Therefore
  $$\sum_q\langle-\rho_q', \xi\rangle =\frac{1}{N}(\deg L) \xi(1)$$
  From Frobenius reciprocity, we obtain
  $$\sum_q\langle\rho_q'', \xi\rangle =m_\xi(L)$$
  This proves \eqref{eq:degVxi}.
\end{proof}

\section{Cyclic and dihedral groups}

We make the Chevalley-Weil formula more explicit  in a couple
of cases. With the previous notation, we calculate the local coefficients $m_{\xi,
  p}(L)$ when $G$ is cyclic or dihedral.

Let $G=\langle g|g^N=1\rangle$. Then $\G$ is also a cyclic group of order $N$
generated by
$$\xi(g^j) = \exp (2 \pi i j/N)$$
The dual of the inclusion $G_p\subseteq G$  gives a surjective
homomorphism $\G\to \G_p$. Since the character $\tau_p$ generates $\G_p$, we can assume  without loss of generality that $\xi|_{G_p}=\tau_p$.
The character $\nu_p$ is given by $\xi^m|_{G_p}$ for nonnegative integer $ m<
 N_p$. Given an integer $x$,  let $\overline{x}\in \{0,1,\ldots
 N_p-1\} $ be its  residue  modulo $N_p$.

\begin{prop}\label{prop:cyclic}
  With the notation above
  $$m_{\xi^k,p}(L) =\frac{ \overline{k-m}}{N}$$
\end{prop}

\begin{proof}
The orthonormality of $\G$ implies that
 \begin{equation}
   \label{eq:innprod}
   \begin{split}
     \langle \xi^k|_{G_p}, \nu_p\cdot \tau_p^i\rangle_{G_p} =
 \begin{cases}
   1 &\text{if $k\equiv i+m\mod N_p$}\\
   0 &\text{otherwise}
 \end{cases}
   \end{split}
 \end{equation}
 Substituting into
 $$m_{\xi^k,p}(L) = \frac{\xi^k(1)}{N} \sum_{i=0}^{N_p-1}
  i\langle \xi^k|_{G_p}, \nu_p\cdot \tau_p^i\rangle_{G_p}$$
proves the proposition.
\end{proof}

Next suppose that $G= D_n$ is the dihedral group of order $N=2n$. This group has a presentation
$$D_n= \langle r,s\mid r^n=1, s^2=1, srs=r^{-1}\rangle$$
Let $R\subset D_n$ be the subgroup generated by $r$.
The characters in $\G$ are easy to write down \cite[\S 5.3]{serre}.
When $n$ is even, 
$$\G= \{\chi_h \mid 0< h<n/2\}\cup \{\psi_i\mid i=1,2,3,4\}$$
where the characters are given by the table
\begin{center}
\begin{tabular}{l|ll} 
  & $r^k$ & $sr^k$ \\ \hline
  $\chi_h$ & $2\cos \frac{2\pi hk}{n}$ & $0$ \\ \hline
 $\psi_1$ & $1$ & $1$ \\ \hline
 $\psi_2$ & $1$ & $-1$ \\ \hline
 $\psi_3$ & $(-1)^k$ & $(-1)^k$ \\ \hline
 $\psi_4$ & $(-1)^k$ & $(-1)^{k+1}$ \\ \hline
\end{tabular}
\end{center}
In more conceptual terms, $\chi_h$ is induced from $\xi^h$, where $\xi$ is a
generator of $\widehat{R}$. We also have $\chi_h|_R = \xi^h+ \xi^{n-h}$.
When $n$ is odd $\psi_3$ and $\psi_4$ are omitted, otherwise $\G$ is
the same.

A  nontrivial cyclic subgroup of $D_n$ is easily seen to be either a
subgroup $R$, or the subgroup $\{1, sr^k\}$
for some $k$. For simplicity, we just treat the case where $G_p\subseteq R$.

\begin{prop}
  Suppose that $G_p\subseteq R$. Lift $\tau_p$ to a generator $\xi\in
  \widehat{R}$, and $\nu_p$ to $\xi^m$. Then
  $$m_{\chi_h,p}(L) =2\left(\frac{ \overline{h-m}+\overline{n-h-m}}{N}\right)$$
  and
  $$m_{\psi_i, p}(L)= \frac{\overline{-m}}{N},\quad i=1,2$$
  When $n$ is even
  $$m_{\psi_i, p}(L)=\frac{\overline{n/2-m}}{N},\,\quad i=3,4$$
\end{prop}

\begin{proof}
From the properties above
$$  \langle \chi_h|_{G_p}, \nu_p\cdot \tau_p^i\rangle_{G_p} =
 \begin{cases}
   1 &\text{if $h\equiv i+m\mod N_p$ or $n-h\equiv i+m\mod N_p$}\\
   0 &\text{otherwise}
 \end{cases}
 $$
 Substituting into the formula for $m_{\chi_h, p}(L)$ yields the first equation.
 The proofs of the  remaining equations are similar.
\end{proof}

\section{The Lefschetz formula}

In this section, we want to explain the relation to
the holomorphic Lefschetz fixed point formula  of Atiyah and
Bott \cite[thm 4.12]{ab}. We refer to their paper for the general
statement, as well  as for some historical information about it.
First let us we briefly digress from Riemann surfaces to explain that
the  conclusion of corollary \ref{cor:free}
is true in much greater generality.

\begin{prop}
 \-
 \begin{enumerate}
\item If a finite group $G$ acts freely on a finite simplicial complex $X$, then
the character  $\chi_G(X)$ of  the virtual representation $\sum (-1)^i H^i(X, \C)$
is an integer multiple of a regular representation.
\item If $G$ acts freely on a compact complex manifold $X$, and $V$ is an equivariant vector bundle, then
the character $\chi_G(V)$ of the virtual representation  $\sum (-1)^i H^i(X,V)$ 
is an integer multiple of a regular representation.
\end{enumerate}

\end{prop}

\begin{proof}
 By the Lefschetz fixed point theorem \cite[p 179]{hatcher} $\chi_G(X)(g)=0$ when $g\not=1$. 
This implies that $\chi_G(X)$ 
  is an integer multiple of the regular representation by \cite[ ex 2.7]{serre}.
  The proof of the second statement is identical except that we use
  the holomorphic Lefschetz fixed point theorem 
  \cite[thm 4.12]{ab}.

\end{proof}

\begin{cor}
In case (1) (resp. 2), if $\chi(X)$ (resp. $\chi(V)$) denotes the topological (resp. holomorphic) Euler characteristic of $X$ (resp. $V$), then
 $$\chi_G(X)=\frac{\chi(X)}{|G|}\chi_{reg}$$
 $$\chi_G(V)=\frac{\chi(V)}{|G|}\chi_{reg}$$
 The coefficients on the right hand side of the two equations are integers.
\end{cor}

\begin{proof}
 We know that $\chi_G(X) = c \chi_{reg}$ (resp. $\chi_G(V)=c
 \chi_{reg}$) for some $c\in \Z$. Evaluation at $g=1$ gives  the value
 of $c$.
\end{proof}

We will deduce the holomorphic
Lefschetz formula for finite order automorphisms of a Riemann surface
from the Chevalley-Weil theorem by a fairly elementary reduction.

\begin{thm}\label{thm:CWcyclic}
  Let $g$ be a nontrivial  finite order automorphism of a compact
  Riemann surface $X$. 
    Let $L$ be a line bundle, which is equivariant for the group generated by $g$, and let $X^g=\{p\mid
  g(p)=p\}$,
  and let $\tau_p,\nu_p$ be as in section 2. Then
  $$\sum_{i=0}^1 (-1)^i\tr \left[g:H^i(X,\cL) \to H^i(X,\cL)\right] = \sum_{p\in X^g}\frac{\nu_p(g)}{1-\tau_p(g)}$$
  \end{thm}

\begin{proof}
Let $G\subseteq Aut(X)$ be the group generated by $g$, and let $N$
denote its order.
 Since $g\not=1$,
 we have  $\chi_{reg}(g)=0$. 
 Therefore  theorem \ref{thm:cw}  reduces to
 $$\chi_G(L)(g) = \sum_{p\in X} M_p$$
 where
 \begin{equation*}
   M_p = -\sum_{\xi\in \G} m_{\xi,p}(L)\xi(g)
  \end{equation*}

Of course $M_p=0$ unless $N_p\not=1$.  Let us assume this.
 Let $n= N/N_p$ denote the
 index of $G_p$ in $G$. Employing the statement and notation of
 proposition~\ref{prop:cyclic}, we find
 \begin{equation*}
  \begin{split}
    M_p &=
    -\frac{1}{N}\sum_{k=m}^{N-1+m}\, \overline{k-m}\,    \zeta^k\\
&=-\frac{1}{N} \sum_{i=0}^{N_p-1} \sum_{\ell=0}^{n-1}i\zeta^{i+m+\ell N_p}\\
  \end{split}
\end{equation*}
where the last expression results from the substitutions
$i=\overline{k-m}$ and $\ell =(k-m-i)/N_p$.
Therefore
\begin{equation}\label{eq:Mp2}
M_p=-\frac{\zeta^m}{N} \left(\sum_{i=0}^{N_p-1} i\zeta^{i}\right)\left(\sum_{\ell=0}^{n-1}(\zeta^{N_p})^\ell\right)
\end{equation}

 We now separate the argument into two cases depending on $n$.
First suppose that $n\not=1$. Then we claim that
$$M_p=0$$
which would be predicted by the theorem because
$X^g=\{p\mid g\in G_p\}=\{p\mid N=N_p\}$.
To see this, observe that $\zeta^{N_p}$ is a primitive $n$th
root of unity, therefore
$$\sum_{\ell=0}^{n-1}(\zeta^{N_p})^\ell=0$$
Substituting this into \eqref{eq:Mp2} shows that $M_p=0$ as claimed.

Let us now assume that $n=1$, which means that  $N_p=N\not=1$ so that $p\in X^g$.
One checks the identity
$$\sum_{i=0}^{N-1} i\zeta^i =\frac{N}{\zeta-1}$$
by cross multiplying and observing
$$\sum_{i=0}^{N-1} i\zeta^i (\zeta-1)=N\zeta^N-\sum_{i=1}^{N}\zeta^i=N$$
Therefore by \eqref{eq:Mp2}
\begin{equation*}
  \begin{split}
    M_p  
    &=-\frac{\zeta^m}{N} \left(\sum_{i=0}^{N-1} i\zeta^{i}\right)\\
    &= \frac{\zeta^m}{1-\zeta}\\
    &= \frac{\nu_p(g)}{1-\tau_p(g)}
  \end{split}
\end{equation*}
   
\end{proof}

\end{document}